\documentclass{article}
\usepackage{hyperref}
\usepackage[utf8]{inputenc}
\usepackage[cyr]{aeguill}
\usepackage[french, english]{babel}
\usepackage[T1]{fontenc}
\usepackage{amscd}  
\usepackage{amsmath}
\usepackage{amsthm} 
\usepackage{amssymb}
\usepackage{titlesec}
\usepackage{enumitem}
\theoremstyle{definition}
\newtheorem{theointro}{Theorem}

\newtheorem{propintro}[theointro]{Proposition}   
\newtheorem{theo}{Theorem}[section]
\newtheorem{prop}[theo]{Proposition}
\newtheorem{lemm}[theo]{Lemma}
\titleformat{\section}
     {\center\Large\scshape}{\thesection.}{1em}{}
 \titleformat{\subsection}[runin]
     {\scshape}{\thesubsection. }{0em}{}
     [. ]
\def\Com{\mathrm{Com}}
\def\Lie{\mathrm{Lie}}
\def\alg{\mathrm{alg}}
\def\d{\mathsf{d}}
\def\D{\mathsf{D}}
\def\Sym{\mathsf{Sym}}
\def\DBun{\mathsf{DBun}}
\def\id{\mathrm{id}}
\def\F{\mathbb F}
\def\X{\mathbb X}
\def\P{\mathcal P}
\def\T{\mathsf T}
\def\S{\mathsf S}
\def\im{\mathsf{im}}
\def\Mod{\mathsf{Mod}}
\def\∑{\Sigma}
\author{}
\title{Tangent structures for divided power algebras}
\author{Sacha Ikonicoff}
\begin{document}
\allowdisplaybreaks

\maketitle
\begin{abstract}
	We build a tangent structure on the category of divided power algebras using a particular notion of semidirect product. We show that this tangent structure admits an adjoint tangent structure, which involves a version of Kähler differentials, and which is similar to the Zariski cotangent space for affine schemes. We study vector fields and differential bundles for these two structures, which correspond respectively to a notion of special derivation, and to the category of modules over the underlying commutative algebra of a given divided power algebra.
\end{abstract}
\tableofcontents
\subsection*{Introduction} Divided power algebras are a class of commutative algebras equipped with additional monomial operations. Briefly, a divided power algebra $A$ is an algebra equipped, for all $n$, with an operation $a\mapsto a^{[n]}$ which behaves like the operation $a\mapsto \frac{a^n}{n!}$. Indeed, in characteristic 0, those operations coincide, so divided power algebras are just usual commutative algebras. However, in positive characteristics, the two notions diverge. These algebras were introduced by Cartan in \cite{cartan1954puissances}, and in this same series of articles, it is shown that the cohomology of Eilenberg--MacLane spaces is equipped with the structure of a divided power algebra. More generally, Cartan shows that the homotopy of a simplicial algebra has divided powers. 

Since then, divided power algebras have appeared as a key notion for different algebraic theories in positive characteristic, notably as a structure on the crystalline cohomology of schemes \cite{berthelot74}, and in deformation theory \cite{fox93}.

The aim of this article is to build certain tangent structures on the category of divided power algebra and its opposite. The notion of a tangent category, due to Rosick\'y \cite{RosickyTangentCats}, and rediscovered by Cockett and Cruttwell \cite{cockett2014differential}, is a categorical formalisation which axiomatises the structure of the tangent bundle functor from differential geometry. As such, a tangent structure on a category governs a certain geometry on the objects of this category. While smooth manifolds and smooth maps between them is the archetypal example of tangent category, examples of a more algebraic flavour have also been studied. In particular, the category of rings admits an interesting tangent structure, whose tangent bundle involves the ring of dual number \cite[Example 2]{RosickyTangentCats}. This tangent structure is studied in more details in \cite[Section 3]{cruttwellLemay:AlgebraicGeometry}, and appears in \cite[Example 4.18]{ikonicoffLL24}, where it is built using a certain notion of semidirect product. Another, perhaps even more interesting tangent structure, appears on the opposite category of commutative rings. This structure, adjoint to the structure involving dual numbers, uses the notion of Kähler differentials, and corresponds to the Zariski cotangent space, simply called \textit{fibré tangent} by Grothendiek \cite[Définition 16.5.12.I]{grothendieck1966elements}. This structure is mentioned in \cite{cockett2014differential,GARNER2018668}, and studied in more details in \cite{cruttwellLemay:AlgebraicGeometry} and \cite[Example 4.26]{ikonicoffLL24}.

In this article, we will use techniques similar to those of \cite{ikonicoffLL24}: indeed, we will build a tangent structure on divided power algebras from a certain differential combinator on the monad $\Gamma$ whose algebras are divided power algebras. This differential combinator was introduced in \cite[Section 6]{ikonicofflemay23}, and uses a natural notion of derivation for divided power monomials (see \cite{KP}). The induced tangent structure involves a certain semidirect product for divided power algebras. We call this structure the algebraic tangent structure for divided power algebras.

We then show that the tangent bundle functor admits a left adjoint, obtained from a certain notion of Kähler differentials for divided power algebras. This implies that the opposite category of divided power algebras admits a tangent structure whose tangent bundle is given by this left adjoint functor. By analogy with the ring of coordinates for an affine scheme, we call the objects of the opposite category of divided power algebras the divided affine scheme. We refer to the adjoint tangent structure involving Kähler differential as the geometric tangent structure for divided affine scheme.

To conclude this article, we study the notions of vector fields and differential bundles in these two tangent categories. The notions of vector fields and differential bundles in a tangent category were respectively introduced in \cite{cockett2015jacobi} and \cite{cockett2021differential}. In \cite[Lemma 2.9 and Example 4.3.3]{ikonicoffLL24}, we show that vector fields in both the tangent structure on commutative rings using dual numbers, and in the adjoint tangent structure, correspond to the usual notion of derivation for a ring. Here, we will show that vector fields in our algebraic and geometric tangent structures for divided power algebras and divided schemes correspond to a certain notion of special derivation, which are derivations $d:A\to A$ of a divided power algebra $A$ satisfying the additional condition $d(a^{[p]})=-a^{p-1}d(a)$, where $p$ is the characteristic of the base field. In \cite{cruttwellLemay:AlgebraicGeometry} Cruttwell and Lemay show that differential bundles over a commutative ring $R$ correspond to $R$-modules in both tangent structures. Here, we will show that tangent bundles over a divided power algebra $A$ will correspond, in both of our tangent structures, to modules over the underlying commutative algebra of $A$, forgetting the divided power structure.

To summarise, here are the main results of this article:
\begin{theointro}[\ref{theo:tanstructdivpow}]
    There is a tangent structure on the category of divided power algebras for which the tangent bundle functor $\T$ sends the divided power algebra $A$ to the semidirect product $A\ltimes_pA$.
\end{theointro}
\begin{theointro}[\ref{theo:tanstructdivschemes}]
    There is a tangent structure on the opposite category of divided power algebras for which the tangent bundle functor $\T$ sends the divided power algebra $A$ to $\Gamma^+_A(\Omega^1_A)$, the free $A$-divided power algebra generated by the strongly abelian Kähler differentials of $A$.
\end{theointro}
\begin{propintro}{\ref{prop:vectorfieldalg} and \ref{sec:vectorfieldsch}}
    There is a one-to-one correspondence between vector fields in the algebraic tangent structure (and in the geometric tangent structure) over $A$ and special inner derivations of $A$.
\end{propintro}
\begin{theointro}{\ref{theo:equivbundlemodulealg}}
    There is an equivalence of categories between the category of differential bundles over $A$ in the algebraic tangent structure and the category of modules over the underlying commutative algebra of $A$.
\end{theointro}
\begin{theointro}{\ref{theo:equivbundlemodulesch}}
    There is an equivalence of categories between the category of differential bundles over $A$ in the geometric tangent structure and the opposite category of modules over the underlying commutative algebra of $A$.
\end{theointro}
\subsection*{Outline} Section \ref{sec:recdivpowealg} and \ref{sec:recdifftan} are background sections dedicated respectively to divided power algebras, and to differential and tangent structures. In \ref{sec:difftandivpowalg}, we combine results from the two first section to define our tangent structure on the category of divided power algebras and its opposite. In Sections \ref{sec:algtanstruc} and \ref{sec:geotanstruc}, we identify vector fields and differential bundles respectively for the algebraic tangent structure of divided power algebras, and for the geometric tangent structure on divided affine schemes. Finally, in Section \ref{sec:futurework}, we present some interesting topics open for future investigation.
\subsection*{Acknowledgment} I would like to thank JS Lemay for reminding me that I had written this.

\subsection*{Notations} We fix a base field $\F$. Unless otherwise specified, all vector spaces and algebras are considered over $\F$. We denote by $\∑_n$ the symmetric group on $n$ letters.
\section{Recollection on divided power algebras}\label{sec:recdivpowealg}
In this section, we recall the classical definition of a divided power algebra, due to Cartan \cite{cartan1954puissances}. We introduce our notation, and study certain notions of modules and derivations, as well as the construction of coproducts and pushouts in the category of divided power algebras and in the slice category of algebras under a given divided power algebra. Our main references for this section are \cite{roby63,roby65alg,soublin87puissances} for the classical notions and \cite{dokastriplecohomology,dokas23} for more modern notions of modules and derivations. Note that we will work with non-unital divided power algebras, and the divided power operations are defined on the entire algebra, whereas in the classical references, one works with a unital algebra, and the divided power structure is defined solely on an ideal - sometimes an augmentation ideal. Our notion is equivalent to the augmented cases via the usual equivalence between non-unital algebras and augmented algebras.
\subsection{Divided power algebras}\label{sec:defipuisdiv} 
    A \textbf{divided power algebra} \cite[Section 4]{cartan1954puissances} is a commutative associative (non-unital) algebra $(A,*)$, equipped with a divided power structure, that is, a family of functions $(-)^{[n]}: A \to A$, $a \mapsto a^{[n]}$, indexed by strictly positive integers $n$, such that the following identities hold:
    \begin{enumerate}[label={\bf [dp.\arabic*]}]
\item\label{relComlambda} $(\lambda a)^{[n]}=\lambda^na^{[n]}$ for all $a\in A$ and $\lambda\in\mathbb{F}$.
\item\label{relComrepet} $a^{[m]}*a^{[n]}=\binom{m+n}{m}a^{[m+n]}$ for all $a\in A$.
\item\label{relComsomme} $(a+b)^{[n]}=a^{[n]}+\big(\sum_{l=1}^{n-1}a^{[l]}*b^{[n-l]}\big)+b^{[n]}$ for all $a\in A$, $b\in A$.
\item\label{relComunit} $a^{[1]}=a$ for all $a\in A$.
\item\label{relComcomp1} $(a*b)^{[n]}=n!a^{[n]}*b^{[n]}=a^{*n}*b^{[n]}=a^{[n]}*b^{*n}$ for all $a\in A$, $b\in A$.
\item\label{relComcomp2} $(a^{[n]})^{[m]}=\frac{(mn)!}{m!(n!)^m}a^{[mn]}$ for all $a\in A$.
\end{enumerate}
The function $(-)^{[n]}$ is called the $n$-th \textbf{divided power operation}. 

We denote by $\Gamma_{\alg}$ the category of divided power algebras. Divided power algebras are algebras over a monad $\Gamma$ described in the next section. We denote by $U(A)$ the underlying commutative algebra of a divided power algebra $A$, forgetting the divided power operations.

Note that, in characteristic zero, one can show that $a^{[n]}=\frac{a^n}{n!}$, so divided power algebras over a field of characteristic zero are the same thing as usual commutative and associative algebras. When the characteristic of the base field $\F$ is $p>0$, Soublin shows \cite{soublin87puissances} that the divided power operations are generated by the $p$-th divided power $a\mapsto a^{[p]}$ satisfying the relations:
    \begin{enumerate}[label={\bf [dpp.\arabic*]}]
\item\label{relComp1} $a^{*p}=0$,
\item\label{relComp2} $(a+b)^{[p]}=a^{[p]}+b^{[p]}+\sum_{i=1}^{p-1}\frac{(-1)^i}{i}a^{*i}*b^{* p-i}$,
\item\label{relComp3} $(a*b)^{[p]}=0$,
\end{enumerate}
for all $a,b\in A$.

We will often assume that the characteristic of the base field is a fixed prime $p$ and use the result above without mentioning it.
\subsection{Free divided power algebras}\label{sec:freedividedpoweralgebras} For all vector spaces $V$, denote by $\Gamma(V)$ the space of symmetric tensors of $V$, that is, $\Gamma(V)=\bigoplus_{n>0}\left(V^{\otimes n}\right)^{\∑_n}$. The vector space $\Gamma(V)$ is endowed with a divided power algebra structure, and is the free divided power algebra over $V$ \cite[Section 2]{cartan1954puissances}. Explicitly, the divided power operations and the product are defined on generators $v,w\in V$ by:
\[
v^{[n]}=v^{\otimes n}\qquad v*w=v\otimes w+w\otimes v.
\]
An arbitrary element of $\Gamma(V)$ can then be expressed as a finite sum of \textbf{divided power monomials} \cite[Section 4]{cartan7relations}, which are elements of the form: 
\[ v_1^{[r_1]}*\hdots* v_s^{[r_s]}= \sum_{\sigma\in \∑(n)/\∑(r_1,\hdots,r_s)}\sigma (v_1^{\otimes r_1}\otimes\hdots\otimes v_s^{\otimes r_s}),\]
where $\∑(r_1,\hdots,r_s) = \∑(r_1)\times\hdots\times \∑(r_s)$ is the Young subgroup of the symmetric group $\∑(n)$, with $n=r_1+\hdots+r_s$.

The assignment $V\mapsto \Gamma(V)$ extends to a monad in vector spaces. The unit $\eta_V:V\to \Gamma(V)$ identifies $V$ with ($V^{\otimes 1})^{\∑_1}\subseteq\Gamma(V)$, so $\eta_V(v)=v^{[1]}$. The multiplication ${\mu_V:\Gamma(\Gamma(V))\to\Gamma(V)}$ comes from the shuffle product, and corresponds to the composition of divided power monomials. Using \textbf{[dp.5]} and \textbf{[dp.6]}, we can express this multiplication on divided power monomials of divided power monomials: 
\begin{multline*}
   \mu_V\left((v_{1,1}^{[q_{1,1}]}*\hdots*v_{1,k_1}^{[q_{1,k_1}]})^{[r_1]}*\hdots*(v_{p,1}^{[q_{p,1}]}*\hdots*v_{p,k_p}^{[q_{p,k_p}]})^{[r_p]}\right)\\
   =~\left(\prod_{i=1}^p\frac{1}{r_i!}\prod_{j=1}^{k_i}\frac{(r_iq_{i,j})!}{q_{i,j}!^{r_i}}\right)v_{1,1}^{[r_1q_{1,1}]}*\hdots*v_{1,k_1}^{[r_1q_{1,k_1}]}*\hdots*v_{p,k_p}^{[r_pq_{p,k_p}]} 
\end{multline*}
which we then extend by linearity.
\subsection{Divided power modules}\label{section:dividedpowermodules}
For a divided power algebra $A$, a \textbf{divided power $A$-module}, which we will simply refer to as $A$-modules when $A$ is understood as a divided power algebra, is a $U(A)$-module $M$ equipped with a function $\pi:M\to M$ called the \textbf{$p$-map} satisfying, for all $a\in A$, $m,m'\in M$ and $\lambda\in\F$:
\begin{enumerate}[label={\bf [dpm.\arabic*]}]
\item\label{rel:dpm1} $\pi(m+m')=\pi(m)+\pi(m')$,
\item\label{rel:dpm2} $\pi(\lambda m)=\lambda^p\pi(m)$,
\item\label{rel:dpm3} $\pi(a\cdot m)=0$.
\end{enumerate}
Conditions \ref{rel:dpm1} and \ref{rel:dpm2} are often called $p$-semilinearity for the map $\pi$. Following \cite[Section 3]{dokastriplecohomology}, adapting to the non-unital case, divided power $A$-modules as described above correspond to Beck modules over $A$. 

\subsection{Semidirect product}
Recall that, from a commutative algebra $A$ and an $A$-module $M$, one may build a commutative algebra $A\ltimes M$ called the \textbf{semidirect product} of $M$ by $A$, with underlying vector space $A\oplus M$, and multiplication given by:
\[
(a,m)*(b,n)=(a*b,a\cdot n+b\cdot m).
\]
Following \cite[Lemma 3.1]{dokastriplecohomology}, if $M$ be a divided power $A$-module, then the semidirect algebra $A\ltimes M$ of $U(A)$ with the $U(A)$-module $M$ obtained by forgetting the $p$-map is equipped with a divided power algebra structure such that:
    \[
    (a,m)^{[p]}=(a^{[p]},\pi(m)-a^{p-1}\cdot m)
    \]
The resulting divided power algebra is denoted by $A\ltimes_p M$. 

For $A$ a divided power algebra, and $M$ a $U(A)$-module, one can consider $M$ as a divided power $A$-module such that $\pi:M\to M$ is trivial. In this case, the divided power algebra $A\ltimes_p M$ satisfies:
    \[
(a,m)^{[n]}=(a^{[p]},-a^{* p-1}\cdot m).
\]

\subsection{Special derivations}\label{section:specialderivations} A \textbf{special derivation} from a divided power algebra $A$ to an $A$-module $M$ is a linear map $D:A\to M$ such that $D(a*b)=a\cdot D(b)+b\cdot D(A)$ and $D(a^{[p]})=\pi(D(a))-a^{*p-1}\cdot D(a)$. When $M$ is equal to $A$ seen as a $U(A)$-module, these correspond to the special derivations of \cite{dokas23}, and will be called \textbf{special inner derivations}. In this case, the second relation reads $D(a^{[p]})=-a^{*p-1}D(a)$ and correspond to the power rule of \cite[Section 2]{KP}. 

Following \cite[Section 3]{dokastriplecohomology}, special derivations correspond to Beck derivations: one can check that derivations $A\to M$ correspond to divided power algebra morphisms $A\to A\ltimes_p M$ whose first coordinate is the identity on $A$.
\subsection{Divided power \texorpdfstring{$A$}{A}-algebras}\label{sec:divpowerA-alg}
Later on, we will consider divided power algebras over a divided power algebra $A$, which are equivalently defined by morphisms of divided power algebras $f:A\to B$. In particular, this implies that $B$ is a divided power algebra with an action of $U(A)$.

For a $U(A)$-module $M$, denote by $\Gamma_A(M)=\bigoplus_{n>0}(M^{\otimes_{U(A)}n})^{\∑_n}$, which is equipped with a divided power algebra structure given by the shuffle product and tensor powers, like in Section \ref{sec:freedividedpoweralgebras}. 

Such divided power $A$-algebras were studied by Roby in \cite{roby65alg} in the unital case. We will refer to results of \cite{roby65alg} \textit{mutatis mutandis} to the non-unital case. Applying \cite[Théorème 2]{roby65alg} the functor sending an $A$-algebra $f:A\to B$ to the $U(A)$-module $B$ admits a left adjoint, which we will denote by $\Gamma^+_A$. If $M$ is a $U(A)$-module, then $\Gamma^+_A(M)$ is given by the $U(A)$-module $A\oplus \Gamma_A(M)$, with the map $f:A\to \Gamma^+_A(M)$ given by the inclusion in the first factor. The multiplication in $\Gamma^+_A(M)$ is given by that of $A$, of $\Gamma_A(M)$ given by shuffle product, and by the action of $U(A)$ on $M^{\otimes_{U(A)} n}$. Using \ref{relComsomme}, we can express the divided power operations by:
\[
(a,\underline{m})^{[k]}=(a^{[k]},0)+(0,\underline{m}^{\otimes k})+\sum_{i=1}^{k-1}(0,a^{[i]}\cdot \underline{m}^{\otimes_{U(A)} k-i}).
\]

\subsection{Coproducts and pushouts}
Adapting \cite[Théorème III.3]{roby63} to the non-unital case, we get a natural linear bijection $\Gamma(V\times W)\cong \Gamma(V)\oplus \Gamma(W)\oplus \Gamma(V)\otimes \Gamma(W)$ for all vector spaces $V,W$. Since $\Gamma$ is a left adjoint functor, it preserves coproducts. This implies that the coproduct $\Gamma(V)\coprod\Gamma(W)$ in the category of divided power algebras has underlying vector space given by $\Gamma(V)\oplus \Gamma(W)\oplus \Gamma(V)\otimes \Gamma(W)$.

Similarly, following \cite[Proposition 9]{roby65alg}, if $M,N$ are two $U(A)$-modules, the coproduct of $\Gamma^+_{A}(M)$ and $\Gamma^+_{A}(N)$ in the category of divided power $A$-algebras has underlying vector space given by $A\oplus \Gamma_A(M)\oplus \Gamma_A(N)\oplus \Gamma_A(M)\otimes_{U(A)}\Gamma_A(N)$.

Let $f:A\to B$ and $g:A\to C$ be two morphisms of divided power algebras. The pushout of $f$ and $g$, which is equivalently the coproduct in the category of divided power $A$-algebras, will be denoted by $[f,g]:A\to B\coprod_A C$. Following \cite[Théorème 3]{roby65alg}, $B\coprod_AC$ has underlying vector space given by $(B\oplus_A C)\oplus B\otimes_{U(A)} C$, where $B\oplus_A C$ is the pushout of $f$ and $g$ in vector spaces.
\section{Recollections on differential and tangent structures}\label{sec:recdifftan}
In this section, we recall two notions coming from differential and tangent categories. We start with the definition of a (Rosick\'y) tangent category, due to Cockett and Cruttwell \cite{cockett2014differential}. We then review the definition of a differential combinator for a monad over a category with biproduct \cite{ikonicoff2021cartesian}, and following \cite{ikonicoffLL24}, we show how such a differential combinator is used to define a tangent structure on the category of algebras over the given monad.

In this section, the base category $\X$ is always assumed to admit finite limits and colimits.
\subsection{Tangent categories}
Following \cite[Definition 2.3 and Section 3.3]{cockett2014differential}, a \textbf{(Rosick\'y) tangent structure} on a category $\mathbb{X}$ is a sextuple $\mathbb{T} := (\mathsf{T}, p, s, z, l, c)$ (resp. a septuple ${\mathbb{T} := (\mathsf{T}, p, s, z, l, c, n)}$) consisting of: 
\begin{enumerate}[label=(\roman*)]
\item An endofunctor $\mathsf{T}: \mathbb{X} \to  \mathbb{X}$, called the \textbf{tangent bundle functor},
\item A natural transformation $p_A: \mathsf{T}(A) \to A$, called the \textbf{projection}, such that for each $n\in \mathbb{N}$, the $n$-fold pullback\footnote{By convention, $\mathsf{T}_0(A) = A$ and $\mathsf{T}_1(A) = \mathsf{T}(A)$} of $p_A$ exists, denoted by $\mathsf{T}_n(A)$ with projections $q_j: \mathsf{T}_n(A) \to \mathsf{T}(A)$, and such that for all $m \in \mathbb{N}$, $\mathsf{T}^m:=\mathsf{T}\circ\dots\circ\mathsf{T}$ preserves these pullbacks, that is, $\mathsf{T}^m( \mathsf{T}_n(A))$ is the $n$-fold pullback of $\mathsf{T}^m(p_A)$ with projections $\mathsf{T}^m(q_j)$,  
\item A natural transformation\footnote{Note that by the universal property of the pullback, it follows that we can define functors $\mathsf{T}_n: \mathbb{X} \to \mathbb{X}$.} $s_A: \mathsf{T}_2(A) \to \mathsf{T}(A)$, called the \textbf{sum},
\item A natural transformation $z_A: A \to \mathsf{T}(A)$, called the \textbf{zero map},
\item A natural transformation $l_A: \mathsf{T}(A) \to \mathsf{T}^2(A)$, called the \textbf{vertical lift},
\item A natural transformation $c_A: \mathsf{T}^2(A) \to \mathsf{T}^2(A)$, called the \textbf{canonical flip}, 
\item (And if Rosick\'y, a natural transformation ${n_A: \mathsf{T}(A) \to \mathsf{T}(A)}$, called the \textbf{negative map},)
\end{enumerate}
such that the equalities in \cite[Definition 2.3]{cockett2014differential} (and if Rosick\'y, also \cite[Definition 3.3]{cockett2014differential}) are satisfied. A \textbf{(Rosick\'y) tangent category} is a pair $(\mathbb{X}, \mathbb{T})$ consisting of a category $\mathbb{X}$ equipped with a (Rosick\'y) tangent structure $\mathbb{T}$ on $\mathbb{X}$.
\subsection{Differential combinator}\label{sec:diffcomb}
Let $\mathbb{X}$ be a semi-additive category and $(\mathsf{S}, \mu, \eta)$ be a monad on $\mathbb{X}$. Following \cite[Example 3.13]{ikonicofflemay23},  A \textbf{differential combinator} on $(\mathsf{S}, \mu, \eta)$ is a natural transformation $\partial_A: \mathsf{S}(A) \to \mathsf{S}(A \times A)$ such that the following equalities hold: 
\begin{description}
\item [{\bf [DC.1]}] $\mathsf{S}(\pi_1) \circ \partial_A = 0$
\item [{\bf [DC.2]}] $\mathsf{S}(\langle \pi_1, \pi_2, \pi_2 \rangle) \circ \partial_A = \mathsf{S}(\langle \pi_1, \pi_2, 0 \rangle) \circ \partial_A + \mathsf{S}(\langle \pi_1, 0, \pi_2\rangle) \circ \partial_A$
\item [{\bf [DC.3]}] $\partial_A \circ \eta_A = \eta_{A \times A} \circ \langle 0, 1_A \rangle$
\item [{\bf [DC.4]}] $\partial_A \circ \mu_A = \mu_{A \times A} \circ \mathsf{S}\left( \mathsf{S}(\langle 1_A, 0 \rangle) \circ \pi_1 + \partial_A \circ \pi_2 \right) \circ \partial_{\mathsf{S}(A)}$
\item [{\bf [DC.5]}] $\mathsf{S}(\langle \pi_1, \pi_4 \rangle) \circ \partial_{A \times A} \circ \partial_A = \partial_A$
\item [{\bf [DC.6]}] $\mathsf{S}\left( \left \langle \pi_1, \pi_3, \pi_2, \pi_4 \right \rangle \right)  \circ \partial_{A \times A} \circ \partial_A = \partial_{A \times A} \circ \partial_A$
\end{description}
\subsection{Induced Tangent structure}\label{section:inducedtangentstructure}
Let $(\mathsf{S}, \mu, \eta)$ be a monad on $\X$. We denote by $\S_\alg$ the category of $\S$-algebras. The objects of $\S_\alg$ are then pairs $(A,\alpha)$ where $A$ is an object of $\X$, and $\alpha:\S(A)\to A$, the structural morphism, satisfies the usual compatibility relations with $\mu_A$ and $\eta_A$. Following \cite[Theorem 3.9]{ikonicoffLL24}, if $\mathsf{S}$ is equipped with a differential combinator $\partial$, then there is a tangent structure on the category on $\S_\alg$ such that:
\begin{enumerate}[label=(\roman*)]
\item The tangent bundle functor is the functor $\mathsf{T}: \S_\alg \to \S_\alg$ defined on objects by:
\begin{align*}
\mathsf{T}(A,\alpha) = \left(A \times A, \left\langle \alpha \circ \mathsf{S}(\pi_1), \alpha \circ \mathsf{S}\left( \pi_1 + \pi_4 \right) \circ \partial_{A \times A} \right \rangle \right),
\end{align*}
and on maps by $\mathsf{T}(f) = f \times f$;
\item The projection is the natural transformation $p_{(A,\alpha)}:\mathsf{T}(A,\alpha) \to (A, \alpha)$ defined by $p_{(A,\alpha)}= \pi_1$, and where the $n$-fold pullback of $p_{(A,\alpha)}$ is:
\begin{multline*}
\mathsf{T}_n(A,\alpha)=\Bigg( \prod\limits^{n+1}_{i=1} A,  \big\langle \alpha \circ \mathsf{S}(\pi_1), \alpha \circ \mathsf{S}\left( \pi_1 + \pi_4 \right) \circ \partial_{A \times A} \circ \mathsf{S}(\langle \pi_1, \pi_2 \rangle), \\\hdots, \alpha \circ \mathsf{S}\left( \pi_1 + \pi_4 \right) \circ \partial_{A \times A}  \circ \mathsf{S}(\langle \pi_1, \pi_{n+1} \rangle) \big\rangle  \Bigg)
\end{multline*}
with pullback projections ${q_j: \mathsf{T}_n(A,\alpha) \to \mathsf{T}(A,\alpha)}$ defined by: $q_j = \langle \pi_1, \pi_{j+1} \rangle$.
\item The sum is the natural transformation $s_A: \mathsf{T}_2(A,\alpha) \to \mathsf{T}(A,\alpha)$ defined by $s_{(A,\alpha)} = \langle \pi_1, \pi_2 + \pi_3 \rangle$.
\item The zero map is the natural transformation $z_{(A,\alpha)}: (A, \alpha) \to \mathsf{T}(A,\alpha)$ defined by $z_{(A,\alpha)} = \langle 1_A, 0 \rangle$.
\item The vertical lift is the natural transformation $l_{(A,\alpha)}:\mathsf{T}^2(A,\alpha) \to \mathsf{T}(A,\alpha)$ defined by $l_{(A,\alpha)} = \langle \pi_1, 0,0, \pi_2 \rangle$.
\item The canonical flip is the natural transformation $c_{(A,\alpha)}: \mathsf{T}^2(A,\alpha) \to \mathsf{T}^2(A,\alpha)$ defined by $c_{(A,\alpha)} = \!\langle \pi_1, \pi_3, \pi_2, \pi_4 \rangle$.
\end{enumerate}
If $\mathbb{X}$ is also an additive category, then this can be upgraded to a Rosick\'y tangent structure where: 
\begin{enumerate}[label=(\roman*)]
\setcounter{enumi}{6}
\item The negative map is the natural transformation $n_{(A,\alpha)}:  \mathsf{T}(A,\alpha) \to \mathsf{T}(A,\alpha)$ defined by $n_{(A,\alpha)} = \langle \pi_1, -\pi_2 \rangle$.
\end{enumerate}
\section{Differential and tangent structures for divided power algebras}\label{sec:difftandivpowalg}
In this section, we combine the notions introduced in Sections \ref{sec:recdivpowealg} and \ref{sec:recdifftan}: we recall the definition of a differential combinator for the monad governing divided power algebras introduced in \cite{ikonicoff2021cartesian}, and we use this combinator to build a tangent structure on the category of divided power algebra (see Theorem \ref{theo:tanstructdivpow}), which, as we show, is given by the semidirect product. We then show that this tangent structure admits an adjoint tangent structure using results of \cite{ikonicoffLL24}. This provides a tangent structure on the opposite of the category of divided power algebras (see Theorem \ref{theo:tanstructdivschemes}). We study this tangent structure, which is given by a certain notion of Kähler differentials for divided power algebras, which resemble the Zariski cotangent space, or Grothendiek \emph{fibré tangent} for affine schemes.
\subsection{Differential combinator for divided power algebras}\label{section:diffcombfordivpowalg} Following \cite[Proposition 6.2]{ikonicofflemay23}, the monad $\Gamma$ is equipped with a differential combinator transformation  $\partial_V:\Gamma(V)\to \Gamma(V\times V)$ such that 
\[\partial_V(v_1^{[r_1]}*\hdots* v_n^{[r_s]})=\sum_{i=1}^s(v_1,0)^{[r_1]}*\hdots* (v_i,0)^{[r_i-1]} *\hdots* (v_s,0)^{[r_s]} * (0,v_i)^{[1]} .\]
Here, when $r_i=1$, the factor $(v_i,0)^{[r_i-1]}$ is omitted.
\subsection{Induced tangent structure}\label{sec:inducedtangentstructure}
\begin{theo}\label{theo:tanstructdivpow}
The category of divided power algebras is equipped with a Rosick\'y tangent structure given by:
    \begin{enumerate}[label=(\roman*)]
    \item A tangent bundle functor $\T:\Gamma_{\alg}\to \Gamma_{\alg}$ defined on objects by $\T: A\mapsto A\ltimes_p A$, where $A\ltimes_p A$ is the semidirect product of $A$ seen as a $U(A)$-module by $A$. Following \cite{ikonicoffLL24}, we denote by $\T^n$ the $n$-th iteration of the functor $\T$. Note that the underlying vector space of $\T^n A$ is isomorphic to $A^{\times 2^n}$.
    \item A natural projection $p_A:\T A\to A$ defined by the first projection $A\ltimes_p A\to A$, $p_A(a,b)=a$, whose $n$-fold pullback is going to be denoted by $\T_n A$. Note that the underlying vector space of $\T_n A$ is isomorphic to $A^{\times n+1}$.
    \item A natural sum map $s_A:\T_2 A\to \T A$, defined by $s_A(a,b,c)=(a,b+c)$.
    \item A natural zero map $z_A:A\to \T A$, defined by $z_A(a)=(a,0)$,
    \item A natural vertical lift $l_A:\T A\to \T^2 A$, defined by $l_A(a,b)=(a,0,0,b)$,
    \item A natural canonical flip $c_A:T^2 A\to T^2 A$ defined by $c_A(a,b,c,d)=(a,c,b,d)$,
    \item A natural negative map $n_A:\T A\to \T A$ defined by $n_A(a,b)=(a,-b)$.
\end{enumerate}
\end{theo}
We will refer to this tangent structure as the \textbf{algebraic tangent structure on divided power algebras}.
\begin{proof}
    We apply \cite[Theorem 3.9]{ikonicoffLL24} (see Section \ref{section:inducedtangentstructure}) to the monad $\Gamma$ equipped with the differential combinator $\partial$ defined in Section \ref{section:diffcombfordivpowalg}. Let $A$ be a divided power algebra and denote by $\alpha:\Gamma(A)\to A$ its structural morphism. We denote the multiplication inside $A$ by concatenation, while we denote the shuffle product in $\Gamma(A)$ by a star, so that $ab=\alpha(a*b)$ for all $a,b\in A$. We denote by $a^{[p]}$ the divided $p$-th power of $a$ in $A$, while we keep $a^{\otimes p}\in \Gamma(A)$, so that $a^{[p]}=\alpha(a^{\otimes p})$.
    The only thing we have to check is that the divided power algebra $\T(A)$ is indeed $A\ltimes_pA$, the rest of the tangent structure is easily derived from \cite[Theorem 3.9]{ikonicoffLL24}.
    
    The divided power algebra $\T(A)$ is defined as the product of vector space $A\times A$ equipped with the structural map 
    \[
    \T(\alpha)=\left\langle \alpha \circ \Gamma(\pi_1), \alpha \circ \Gamma\left( \pi_1 + \pi_4 \right) \circ \partial_{A \times A} \right \rangle:\Gamma(A\times A)\to A\times A.
    \]
    We use the same convention as above, using concatenation and $(-)^{[p]}$ for the divided power algebra structure of $\T(A)$ and $*$ and $(-)^{\otimes p}$ for the divided power algebra structure on $\Gamma(A\times A)$ coming from the shuffle product. Let $(a_1,b_1),(a_2,b_2)\in A\times A$. Then, on one hand:
    \begin{equation*}
        \alpha \circ \Gamma(\pi_1)((a_1,b_1)*(a_2,b_2))=(\alpha(a_1*a_2))=a_1a_2,
    \end{equation*}
    and on the other hand,
    \begin{multline*}
        \alpha \circ \Gamma\left( \pi_1 + \pi_4 \right) \circ \partial_{A \times A}((a_1,b_1)*(a_2,b_2))\\=\alpha\circ\Gamma(\pi_1+\pi_4)((a_1,b_1,0,0)*(0,0,a_2,b_2)+(a_2,b_2,0,0)*(0,0,a_1,b_1)\\
        =\alpha(a_1*b_2+a_2*b_1)=a_1b_2+a_2b_1,
    \end{multline*}
    so in $\T(A)$, $(a_1,b_1)(a_2,b_2)=(a_1a_2,a_1b_2+a_2b_1)$.

    For the divided $p$-th power, let $(a,b)\in A\times A$. On one hand,
    \[
    \alpha\circ \Gamma(\pi_1)\left((a,b)^{\otimes p}\right)=\alpha\left(a^{\otimes p}\right)=a^{[p]},
    \]
    and on the other hand,
    \begin{align*}
        \alpha \circ \Gamma\left( \pi_1 + \pi_4 \right) \circ \partial_{A \times A}\left((a,b)^{\otimes p}\right)&=\alpha\circ\Gamma(\pi_1+\pi_4)\left((a,b,0,0)^{\otimes p-1}*(0,0,a,b)\right)\\
        &=\alpha\left(a^{\otimes p-1}*b\right)\\
        &=\alpha\left(\frac{1}{(p-1)!}a^{*p-1}*b\right)=-a^{p-1}b,
    \end{align*}
    so in $\T(A)$, $(a,b)^{[p]}=\left(a^{[p]},-a^{p-1}b\right)$, and so finally, $\T(A)=A\ltimes_pA$.
\end{proof}
\subsection{Adjoint tangent structure}\label{sec:adjtanstruct}
The category of divided power algebras is complete and cocomplete. In particular, it admits reflexive coequalisers, and so, following \cite[Theorem 3.11]{ikonicoffLL24}, the above tangent structure admits an adjoint tangent structure, and the opposite category of divided power algebras is equipped with a Cartesian Rosick\'y tangent structure. The aim of this section is to identify this adjoint tangent structure.

Let $A$ be a divided power algebra, whose product will be denoted by concatenation $(a,b)\mapsto ab$, and whose divided power operations will be denoted by $a\mapsto a^{[n]}$. Denote by $U(A)$ the underlying commutative algebra, and $\Omega_{U(A)}$ its module of Kähler differentials. Recall that $\Omega_{U(A)}$ is the quotient of the free $U(A)$-module generated by symbols $\d a$ for $a\in A$ under the relations:
$$\d(a+\lambda b)=\d a+\lambda\d b\qquad \d(ab)=a\d b+b\d a\qquad \forall a,b\in A,\ \forall \lambda\in\F.$$
Define $\Omega_A^1$ to be the quotient of $\Omega_{U(A)}$ by the submodule generated by the relations $\d \left(a^{[n]}\right)-a^{[n-1]}\d a$\footnote{Comparing with \cite[Section 4.1]{dokastriplecohomology}, this corresponds to the quotient of the module $\Omega^A_p$ by all elements of the form $Pda$.}. We call $\Omega^1_A$ the module of \textbf{strongly abelian Kähler differentials} of $A$. Denote by $\T^\circ(A)=\Gamma^+_A(\Omega_A^1)$, where $\Gamma^+_A$ was defined in Section \ref{sec:divpowerA-alg}.

We will denote by $(a,b)\mapsto a*b$ the multiplication of $\T^\circ(A)$ and by $a\mapsto \gamma_n(a)$ its $n$-th divided power. From what precedes, $\T^\circ(A)$ has three types of generators, of type $a$, $\d a$ and $a\d b$, for $a,b\in A$, satisfying the following relations, for all $a,a',b,b',c\in A$ and $\lambda\in\F$:
\[
    \d (a+\lambda b)=\d a+\lambda \d b\qquad \d (ab)=a\d b+b\d a\qquad \d \left(a^{[n]}\right)-a^{[n-1]}\d a
\]
\[
    a*b=ab\qquad\gamma_n(a)=a^{[n]}
\]
\[
a\d b*a'\d b'=aa'*\d b*\d b'\qquad  a*\d b=a\d b\qquad a*b\d c=(ab)\d c=ab*\d c
\]
The relation $a*db=adb$ shows that $\T^\circ A$ is in fact generated by $a$ and $da$ for $a\in A$.

Note that the second iteration ${\T^\circ}^2 A$ of $\T^\circ $ on $A$ will have generators of type:
\[a,\quad \d a,\quad \d' a, \quad \d'\d a,\]
for $a\in A$.
\begin{prop}
    The functor $\T^\circ$ is left adjoint to the functor $\T$ defined above by the semidirect product.
\end{prop}
\begin{proof}
    Let $A,B$ be two divided power algebras, whose product will be denoted by concatenation and divided powers will be denoted by $a\mapsto a^{[n]}$. We shall show that there is a bijection, natural in $A$ and $B$, between the set of morphisms of divided power algebras $\Gamma^+_A(\Omega_{U(A)})\to B$ and the set of morphisms of divided power algebras $A\to B\ltimes_p B$.

    For clarity, we will denote by $(a,b)\to a*b$ and $a\mapsto\gamma_n(a)$ the product and divided powers in $\T^\circ(A)$ and $B\ltimes_p B$.

    Let $f:\Gamma^+_A(\Omega_A^1)\to B$ be a morphism of divided power algebras. Using the decomposition $\Gamma^+_A(\Omega_A^1)=A\oplus \Gamma_A(\Omega_A^1)$, the map $f$ corresponds to the sum of a map $f_0:A\to B$ and a map $f_1:\Gamma_A(\Omega_A^1)\to B$. Consider the map $f^\flat:A\to B\ltimes_p B$ sending $a\in A$ to $(f_0(a),f_1(\d a))$.

    Let us show that $f^\flat$ is a morphism of divided power algebras. Let $a,b\in A$. On one hand,
    \begin{align*}
        f^\flat(ab)&=(f_0(ab),f_1(\d (ab)))\\
        &=(f_0(a)f_0(b),f_1(a\d b)+f_1(b\d a))\\
        &=(f_0(a)f_0(b),f_0(a)f_1(\d b)+f_0(b)f_1(\d a))\\
        &=(f_0(a),f_1(b))*(f_0(b),f_1(b))\\
        &=f^\flat(a)*f^\flat(b).
    \end{align*}
    On the other hand,
    \begin{align*}
        f^\flat(a^{[n]})&=(f_0(a^{[n]}),f_1(\d (a^{[n]})))\\
        &=(f_0(a)^{[n]},f_1(a^{[n-1]}\d a))\\
        &=(f_0(a)^{[n]},f_0(a^{[n-1]})f_1(\d a))\\
        &=(f_0(a)^{[n]},f_0(a)^{[n-1]}f_1(\d a))\\
        &=\gamma_n(f_0(a),f_1(\d a))\\
        &=\gamma_n(f^\flat(a)),
    \end{align*}
    So $f^\flat$ is indeed a morphism of divided power algebras.

    Conversely, let $g:A\to B\ltimes_p B$ be a morphism of divided power algebras, which we decompose as a sum $g_0+g_1$ with $g_0,g_1:A\to B$. Note that $g_0$ is a map of divided power algebras, making $B$ into a $U(A)$-module, and that $g_1$ satisfies $g_1(ab)=g_0(a)g_1(b)+g_0(b)g_1(a)$ and $g_1(a^{[n]})=g_0(a^{[n-1]})g_1(a)$. We define a map $g^\sharp:\Gamma^+_A(\Omega_A^1)\to B$ by setting $g^\sharp_0=g_0:A\to B$, and extending $g^\sharp_1(\d a)=g_1(a)$ first into a morphism of $U(A)$-modules $\Omega_A^1\to B$, then into a morphism of divided power algebras $g^\sharp_1:\Gamma_A(\Omega_A^1)\to B$. To ensure that $g^\sharp_1$ is well defined, note that $g^\sharp_1(\d (ab))=g_1(ab)=g_0(a)g_1(b)+g_0(b)g_1(a)=g^\sharp_1(a\d b+b\d a)$, and that $g^\sharp_1(\d (a^{[n]}))=g_1(a^{[n]})=g_0(a^{[n-1]}g_1(a)=g^\sharp_1(a^{[n-1]}\d a)$.

    Now, consider the sum $g^\sharp:\Gamma^+_A(\Omega_A^1)\to B$ of the maps  $g^\sharp_0:A\to B$ and $g^\sharp_1:\Gamma_A(\Omega_A^1)\to B$. Since the divided power structures on $\Gamma^+_A(\Omega_A^1)$ is entirely determined by that of $A$ and $\Gamma_A(\Omega_A^1)$, it is clear that $g^\sharp$ is a morphism of divided power algebras.

    Now, to show that the assignments $f\mapsto f^\flat$ and $g\mapsto g^\sharp$ are mutually inverse, note, first, that a map of divided power algebras $\Gamma^+_A(\Omega_A^1)\to B$ is entirely determined by the image of the elements $a\in A$ and $\d a\in\Omega_A^1$. Now, $(f^\flat)^\sharp_0=f^\flat_0=f_0$, and $(f^\flat)^\sharp_1(\d a)=(f^\flat)_1(a)=f(\d a)$. So $(f^\flat)^\sharp=f$. Inversely, $(g^\sharp)^\flat(a)=(g^{\sharp}_0(a),g^\sharp_1(\d a))=(g_0(a),g_1(a))=g(a)$, so $(g^\sharp)^\flat=g$.
\end{proof}
From what precedes, the unit $\eta_A:A\to \T\T^\circ A=\T^\circ A\ltimes_p \T^\circ A$ is given by $(\id_{\T^\circ A})^\flat$, which sends $a\in A$ to $(a,\d a)$, and the counit $\varepsilon_A:\T^\circ \T A=\T^\circ(A\ltimes_p A)\to A$ is given by $(\id_{A\ltimes_p A})^\sharp$, which is defined on generators by:
\[
\varepsilon_A(a,a')=a \qquad \varepsilon_A(\d (a,a'))=a'.
\]
As a corollary, and following \cite[Proposition 5.17]{cockett2014differential}, we may now identify the adjoint tangent structure:

\begin{theo}\label{theo:tanstructdivschemes}
    The opposite category $\Gamma_{\alg}^{op}$ of the category of divided power algebras is equipped with a Rosick\'y tangent structure given by:
\begin{enumerate}[label=(\roman*)]
    \item The tangent bundle functor $\T^\circ:\Gamma_{\alg}^{op}\to \Gamma_{\alg}^{op}$ defined above. We denote by $(\T^\circ)^n$ the $n$-th iteration of the functor $\T$.
    \item A natural projection whose opposite $p^\circ_A:A\to \T^\circ A$ is defined by $p^\circ_A=p_{\T^\circ A}\circ \eta_A$, $p^\circ_A(a)=a$, the $n$-fold pushout of which is going to be denoted by $\T^\circ_n A$ and is equal to $(\T^\circ A)^{\coprod_A n}$, where $\coprod_A$ is the coproduct of divided power algebras over $A$. Note in particular that $\T^\circ_2 A$ is generated by elements $a$, $\d_1 a$ from the first summand, and $\d_2a$ from the second summand.
    \item A natural sum map whose opposite $s_A:\T^\circ A\to \T^\circ_2 A$ is defined by
    \[s_A^\circ(a)=a,\qquad s_A^\circ(\d a)=\d_1a+\d_2a.\]
    \item A natural zero map whose opposite $z^\circ_A:\T^\circ A\to A$, is defined by the projection on $A$:
    \[
    z^\circ_A(a)=a,\qquad z^\circ_A(\d a)=0,
    \]
    \item A natural vertical lift whose opposite $l_A^\circ:{\T^\circ}^2 A\to \T A$, defined by:
        \[l_A^\circ(a)=a,\quad l_A^\circ(\d a)=l_A^\circ(\d' a)=0, \quad l_A^\circ(\d'\d a)=\d a,\]
    \item A natural canonical flip whose opposite $c^\circ_A:{\T^\circ}^2 A\to {\T^\circ}^2 A$ is defined by:
    \[c_A^\circ(a)=a,\quad c_A^\circ(\d a)=\d' a,\qquad c_A^\circ(\d' a)=\d a, \quad c_A^\circ(\d'\d a)=\d'\d a,\]
    \item a natural negative map whose opposite $n^\circ_A:\T^\circ A\to \T^\circ A$ is defined by:
    \[n^\circ_A(a)=a,\quad n^\circ_A(\d a)=-\d a.\]
\end{enumerate}
\end{theo}
We will refer to the objects of the opposite category of divided power algebras as \textbf{divided affine schemes}, by analogy with the category of affine schemes, seen as the opposite category of the category of commutative rings as their rings of coordinates. We will then refer to the tangent structure above as the \textbf{geometric tangent structure on divided affine schemes}.

From what precedes, we can formulate a universal property for the divided power algebra $\T^\circ A$. Note that the universal derivation $\d:U(A)\to \Omega_{U(A)}$ induces a special derivation $\d:A\to \T^\circ A$, for $\T^\circ A$ seen as a divided $A$-module with trivial $p$-map. Then, consider a morphism of divider power algebras $f:A\to B$. The universal property reads as follows: for any special derivation $D:A\to B$ for $B$ seen as a divided power $A$-module with trivial $p$-map, there is a unique morphism of divided power algebras $g:\T^\circ(A)\to B$ such that $g(a)=f(a)$ and $g(\d a)=D(a)$.
\section{On the algebraic tangent structure}\label{sec:algtanstruc}
In this section, we study certain geometric concepts – namely, vector fields and differential bundles – induced by the tangent structure on divided power algebras defined in Section \ref{sec:inducedtangentstructure}.
\subsection{Vector fields and special derivations}\label{sec:vectorfieldsalg}
By definition \cite[Definition 3.1]{cockett2014differential}, a \textbf{vector field} on a divided power algebra $A$ in the algebraic tangent structure is a morphism $v:A\to \T A$ of divided power algebras which is a section of $p_A$.

\begin{prop}\label{prop:vectorfieldalg}
    There is a one-to-one correspondence between vector fields in the algebraic tangent structure over $A$ and special inner derivations of $A$.
\end{prop}
\begin{proof}
    A vector field on $A$ is a morphism of divided power algebras $v:A\to A\ltimes_p A$ whose first coordinate $v_0:A\to A$ is the identity. As observed in Section \ref{section:specialderivations}, and following \cite[Section 3]{dokastriplecohomology} these correspond precisely to special inner derivations of $A$.
\end{proof}
\subsection{Differential bundles and modules}

A \textbf{differential bundle} over a divided power algebra $A$ is the data of a morphism of divided power algebras $q:E\to A$ equipped with the structure of a commutative monoid in the slice category over $A$ and with a lift map $\lambda:E\to \T E$ satisfying certain relations. We refer the reader to \cite{MacAdamVectorBundles} for the full definition of differential bundles in a tangent category. We will denote by $\DBun_A$ the category of differential bundles over a divided power algebra $A$.

In \cite[Proposition 6 and Corollary 3]{MacAdamVectorBundles}, MacAdam showed that, in a Rosick\'y tangent category, differential bundles $q:E\to A$ are in fact abelian group objects in the slice category of divided power algebras over $A$, equipped with a lift map.

Abelian group objects in the slice category over a divided power algebra $A$ are also known as Beck $A$-modules. As observed in Section \ref{section:dividedpowermodules}, these correspond to our notion of divided power $A$-modules.

We get the following:
\begin{theo}\label{theo:equivbundlemodulealg}
    There is an equivalence of categories:
    $$\ker:\DBun_A\rightleftarrows U(A)_{\Mod}:A\ltimes_p -$$
\end{theo}
\begin{proof}
    Following \cite[Theorem 3.13]{cruttwellLemay:AlgebraicGeometry}, there is an equivalence of categories:    $$\ker:\DBun_{U(A)}\rightleftarrows U(A)_{\Mod}:U(A)\ltimes -,$$
    between the categories of differential bundles over the underlying commutative algebra $U(A)$ of $A$ and $U(A)$-modules in the usual sense.

    Let $M$ be a $U(A)$-module. We equip $M$ with a derived $A$-module structure with trivial $p$-map. The underlying commutative algebra of the resulting divided power algebra $A\ltimes_p M$ is precisely $U(A)\ltimes M$. The divided $p$-th power of $A\ltimes_p M$ is given by $(a,m)^{[p]}=(a^{[p]},-a^{* p-1}\cdot m)$. To show that $A\ltimes_p M$ is equipped with the structure of a differential bundle over $A$, it suffices to show that the structure maps giving $U(A)\ltimes M$ the structure of a differential bundle over $A$ are compatible with the divided $p$-th power. The fact that $M$ with trivial $p$-map is a Beck $A$-module ensures that the map giving $U(A)\ltimes M$ the structure of an abelian group object in the slice category over $U(A)$ are compatible with the divided $p$-th power. We have to show that the lift map $\lambda:U(A)\ltimes M\to \T(U(A)\ltimes M)$ is compatible with the divided $p$-th power. 

    Still following the proof of \cite[Theorem 3.13]{cruttwellLemay:AlgebraicGeometry}, the lift is given by $\lambda(a,m)=(a,0,0,m)$, so,
    \[
    \lambda((a,m)^{[p]})=\lambda(a^{[p]},-a^{p-1}m)=(a^{[p]},0,0,-a^{* p-1}\cdot m).
    \]
    On the other hand, inspecting the divided power structure on $\T(A\ltimes_p M)$, one has:
    \[
    (a,0,0,m)^{[p]}=((a,0)^{[p]},-(a,0)^{*p-1}\cdot (0,m))=(a^{[p]},0,0,-a^{*p-1}\cdot m),
    \]
    and so, we get a morphism of divided power algebras $\lambda: A\ltimes_p M\to \T(A\ltimes_p M)$.

    Conversely, let $q:E\to A$ be equipped with the structure of a differential bundle over $A$. Then, $\ker(q)$ is a divided $A$-module and $E\cong A\ltimes_p \ker(q)$. To finish the proof, it suffices to show that the $p$-map $\pi$ of $\ker(q)$ is trivial. Note that, since $\lambda$ is a morphism of divided power algebras $A\ltimes_p \ker(q)\to \T(A\ltimes_p \ker(q))$, one has, for all $a\in A$ and $m\in \ker(q)$,
    \begin{align*}
        (a^{[p]},0,0,\pi(m)-a^{p-1}m)&=\lambda((a,m)^{[p]}),\\
        &=\lambda(a,m)^{[p]},\\
        &=(a,0,0,m)^{[p]},\\
        &=(a^{[p]},0,0,-a^{p-1}m),
    \end{align*}
    and so, $\pi(m)=0$.
\end{proof}
\section{On the geometric tangent structure}\label{sec:geotanstruc}
The geometric tangent structure on divided affine schemes is the tangent structure defined in Section \ref{sec:adjtanstruct} on the opposite category of divided power algebras. In this section, we study the notions of vector field and differential bundles for this tangent structure.
\subsection{Vector fields and special derivations}\label{sec:vectorfieldsch} Following \cite[Lemma 2.9]{ikonicoffLL24}, if a tangent category has adjoint tangent structure, then there is a one-to-one correspondence between vector fields over an object $A$ for the tangent structure, and vector fields over $A$ in the opposite category for the adjoint tangent structure. In consequence, according to Section \ref{sec:vectorfieldsalg}, vector fields over a divided affine scheme $A$ are in one-to-one correspondence with special inner derivations $A\to A$.
\subsection{Differential bundles and modules} Following \cite[Theorem 4.17]{cruttwellLemay:AlgebraicGeometry}, there is an equivalence between the category of differential bundles over an affine scheme $R$ and the category of $R$-modules. Here, we will prove a very similar result regarding divided affine scheme, following an analogous argument.

Recall that a differential bundle over $A$ in the tangent category of divided affine scheme is a morphism of divided power algebra $q:A\to E$ which is equipped with the structure of an abelian cogroup object in the coslice category under $A$, and with a lift map $\lambda:\T^\circ E\to E$ satisfying certain relations. We again refer to \cite{MacAdamVectorBundles} for the full definition. We denote by $\DBun^\circ_A$ the category of differential bundles over a divided affine scheme $A$.

On one hand, let us show how to extract a $U(A)$-module from a differential bundle:
\begin{lemm}
    Let $q: A\to E$ be a differential bundle over $A$ with lift $\lambda:\T^\circ E\to E$. Then, the image $\im(\D_\lambda)$ of the linear map $\D_\lambda:\lambda\circ \d:E\to E$, where $\d:E\to \T^\circ E$ is induced by the universal derivation, is equipped with a $U(A)$-module structure such that $a\cdot \D_\lambda(e)=\D_\lambda(q(a)e)$.
\end{lemm}
\begin{proof}
    Here we paraphrase the proof of \cite[Lemma 4.9]{cruttwellLemay:AlgebraicGeometry}.
    The morphism $q:A\to E$ makes $E$ into a $U(A)$-module. To show that $\im(\D_\lambda)$ is a $U(A)$-submodule of $E$, it is enough to show that $\D_{\lambda}$ is compatible with the $U(A)$-action. Let $a\in A$, $e\in E$, then one has:
    \[
    \D_\lambda(q(a)e)=\lambda\circ\d(q(a)e)=(\lambda\circ\d\circ q(a)) (\lambda\circ p(e))+(\lambda\circ p\circ q(a))(\lambda\circ \d(e)),
    \]
    where $p:E\to T^{\circ}E$ is the structural projection map. Note that $\d(q(a))=T^{\circ}(q)(\d a)$. One of the axioms for $q:A\to E$ to be a differential bundle reads $\lambda\circ T^\circ(q)=q\circ z:T^{\circ} A\to E$. This implies that $\lambda\circ\d\circ q(a)=\lambda\circ T^{\circ}(q)(\d a)=q\circ z(\d a)$. But $z(\d a)=0$ by definition of $z$. Moreover, $\lambda\circ T^\circ(q)=q\circ z$ also implies that $\lambda\circ p\circ q(a)=\lambda \circ T^{\circ}(q)(a)=q\circ z(a)=q(a)$, so, $\D_\lambda(q(a)e)=q(a)\D_\lambda(e)$, which concludes the proof.
\end{proof}
On the other hand, given a $U(A)$-module $M$, we show that $\Gamma^+_A(M)$ is equipped with a differential bundle structure over $A$. Recall from Section \ref{sec:divpowerA-alg} that $\Gamma^+_A(M)$ is the free divided power algebra over $A$ generated by $M$. Then consider:
\begin{enumerate}[label=(\roman*)]
    \item The injection $q:A\to \Gamma^+_A(M)$. Still according to Section \ref{sec:divpowerA-alg}, the pushout of $2$ copies of $q$ is given by:
    \[
        \Gamma_A^+(M)^{\coprod_A 2}=A\oplus \Gamma_A(M)^{\oplus 2}\oplus \Gamma_A(M)^{\otimes 2}.
    \]
    The elements of $\Gamma_A^+(M)^{\coprod_A 2}$ are then of $4$ types, we denote by $a=q(a)$ for $a\in A$, we denote by $m_1$ (resp. $m_2$) the generators of the first (resp. second) copy of $\Gamma_A(M)$ for $m\in M$, and $m\otimes n$ the pure tensors of generators in $\Gamma_A(M)^{\otimes 2}$. Then, the two structural injections of the pushout $\iota_1,\iota_2:\Gamma_A^+(M)\to \Gamma_A^+(M)^{\coprod_A 2}$ are induced by $\iota_1(a)=\iota_2(a)=q(a)$ and $\iota_1(m)=m_1$, $\iota_2(m)=m_2$.
    \item The co-sum map $\sigma:\Gamma_A^+(M)\to \Gamma_A^+(M)^{\coprod_A 2}$ given by $\iota_1+\iota_2$,
    \item The co-zero map $\zeta:\Gamma_A^+(M)\to A$ given by $\zeta(a)=a$ and $\zeta(m)=0$.

    To define the lift, we need more insight on $\T^\circ(\Gamma^+_A(M))$. Note that, as a divided power algebra, $\T^\circ(\Gamma^+_A(M))$ has four types of generators, $a,\d a,m$, and $\d m$, for $a\in A$ and $m\in M$, under the relations $\d (a+\lambda b)=\d a+\lambda \d b$, $\d(ab)=a\d b+b\d a$, $\d(a^{[n]})=a^{[n-1]}\d a$, $\d(m+\lambda n)=\d m + \lambda\d n$, and $\d(am)=a\d m+ m\d a$.
    \item The lift $\lambda:\T^\circ (\Gamma_A^+(M))\to \Gamma_A^+(M)$ induced by:
    \[
    \lambda(a)=a\quad \lambda(m)=0\quad \lambda(\d a)=0\quad \lambda(\d m)=m.  
    \]
    To check that $\lambda$ is well defined, observe for example that:
    \[
    \lambda(a\d m+ m\d a)=\lambda(a)\lambda(\d m)+\lambda(m)\lambda (\d a)=am=\lambda(\d(am)).
    \]
\end{enumerate}
Then, we get the following:
\begin{lemm}
    The co-sum and co-zero map $\sigma$, $\zeta$ equip $q:A\to \Gamma_A^+(M)$ with the structure of an abelian cogroup in the coslice category over $A$. The lift $\lambda$ then equips $q$ with the structure of a differential bundle over $A$.
\end{lemm}
\begin{proof}
    The proof is the same, \textit{mutatis mutandis}, as that of \cite[Lemma 4.11]{cruttwellLemay:AlgebraicGeometry}.

    In \cite[Lemma 4.11]{cruttwellLemay:AlgebraicGeometry}, the authors build a structure of differential bundles with negatives, adding a negative map to the structure, which we can define as $\nu:\Gamma_A^+(M)\to \Gamma_A^+(M)$ by $\nu(a)=a$, $\nu(m)=-m$.

    Note that, in the proof of \cite[Lemma 4.11]{cruttwellLemay:AlgebraicGeometry}, the proof that Diagram (12) is a pushout relies on the fact that the symmetric $R$-algebra over $M$ is the free commutative $R$-algebra generated by $M$, which in our case is replaced by the fact that $\Gamma^+_A(M)$ is the free divided power $A$-algebra generated by $M$. The rest of the computation is done on generators, assuming that the morphisms involved are ring morphisms. In our case, the same computations hold, assuming the morphisms involved are divided power algebra morphisms.
\end{proof}
We finally get:
\begin{theo}\label{theo:equivbundlemodulesch}
    There is an equivalence of categories:
    $$\im(\D_\lambda):\DBun^\circ_A\rightleftarrows U(A)_{\Mod}^{op}:\Gamma^+_A.$$
\end{theo}
\begin{proof}
    The proofs of \cite[Lemmas 4.13 and 4.14, Theorem 4.17]{cruttwellLemay:AlgebraicGeometry} can again be adapted to the divided power algebra case, replacing the commutative ring $R$ by the divided power algebra $A$, the symmetric algebra $\Sym_R(M)$ by our $\Gamma^+_A(M)$, and so on. While these proofs are quite involved, we do not think it would be enlightening to paraphrase them here. Note that, in the proof of \cite[Lemma 4.13]{cruttwellLemay:AlgebraicGeometry}, the morphism $\psi^{-1}_{\mathcal E}:E\to \Sym_R(\im(\D_\lambda))$ has been unfortunately typed as $\psi^{-1}_{\mathcal E}:\ker(\mathsf q)[\varepsilon]\to E$: this is a typo.
\end{proof}
\section{Future Work}\label{sec:futurework}
\subsection{Restricted Lie algebras, divided power algebras over an operad} In \cite{dokasFI2025}, we studied the Quillen cohomology of divided power algebras over an operad, using generalised notions of semidirect products and module of Kähler differentials for divided power algebras over a reduced operad $\P$. When $\P=\Com$ is the operad of commutative, associative algebra, we recover the notions from \cite{dokastriplecohomology,dokas23} and presented in this article for divided power algebras in the classical sense. This suggests that the generalised semidirect product and Kähler differentials may provide tangent structures for the category divided power $\P$-algebra and its opposite, for any reduced operad $\P$. However, to follow the same argument as in the present article, one would need to prove that the monad $\Gamma(\P)$ is equipped with a differential combinator which represents derivations in the divided power setting. While there is a reasonable candidate for such a differential combinator, it remains to show that this candidates conditions \textbf{[DC.1]} to \textbf{[DC.6]} from Section \ref{sec:diffcomb}. In particular, the chain rule \textbf{[DC.4]} is not obviously satisfied.

One may want to start by studying $\Gamma(\Lie)$-algebras, where $\Lie$ is the operad governing Lie algebras. Those are also known as reduced Lie algebras, and their Quillen cohomology was studied in detail in \cite{dokas04}.
\subsection{Differential equations in divided power algebras}
In an unpublished paper \cite{bakerdivpow}, Baker studies certain differential equations in the setting of divided power formal power series, as a natural extension of the classical theory of ordinary differential equations for formal power series. Using the formalism of \cite{cockett2021differential}, our tangent structures on divided power algebras and divided affine schemes offer another approach to differential equations on divided power structures. Studying the possible similarities and differences between these two approaches would bring valuable insight on the calculus of divided power algebras.

\subsection{Crystalline cohomology}
Divided power structures appear notably on the crystalline cohomology of schemes \cite{berthelot74}. This cohomology is related to the de Rham cohomology of schemes. Cruttwell and Lucyshyn-Wright have developed a general notion of de Rham cohomology for objects of a tangent category \cite{CruttwellLucychynCohomology}. As noted at the end of the latter article, it is not clear whether the De Rham cohomology coming from the geometric tangent structure of affine schemes corresponds to the classical de Rham cohomology on schemes. So far, in the theory of tangent categories, there is no analogue for crystalline cohomology. One may wonder if the de Rham cohomology for the geometric tangent structure of affine schemes allows to recover their crystalline cohomology. The approach in the present article is different: we start with divided power algebras, for which we create a tangent structure akin to that of affine scheme. Then, the natural question becomes: what is the appropriate analogue of crystalline cohomology for divided affine schemes?


\begin{thebibliography}{10}

\bibitem{bakerdivpow}
A.~Baker.
\newblock Differential equations in divided power algebras.
\newblock {\em Glasgow University Mathematics Department preprint 95/19}.

\bibitem{berthelot74}
Pierre Berthelot.
\newblock {\em Cohomologie cristalline des sch\'emas de caract\'eristique
  {$p>0$}}, volume Vol. 407 of {\em Lecture Notes in Mathematics}.
\newblock Springer-Verlag, Berlin-New York, 1974.

\bibitem{cartan1954puissances}
H.~Cartan.
\newblock Puissances divis{\'e}es.
\newblock {\em S{\'e}minaire Henri Cartan}, 7(1):1--11, 1954.

\bibitem{cartan7relations}
H~Cartan.
\newblock Relations entre les op{\'e}rations pr{\'e}c{\'e}dentes et les
  op{\'e}rations de {B}ockstein; alg{\`e}bre universelle d'un module libre
  gradu{\'e}.
\newblock {\em S{\'e}minaire Henri Cartan}, 7(1):1--9, 1954.

\bibitem{cockett2014differential}
R.~Cockett and G.~Cruttwell.
\newblock Differential structure, tangent structure, and {SDG}.
\newblock {\em Applied Categorical Structures}, 22(2):331--417, 2014.

\bibitem{cockett2015jacobi}
R.~Cockett and G.~Cruttwell.
\newblock The jacobi identity for tangent categories.
\newblock {\em Cahiers de Topologie et G{\'e}om{\'e}trie Diff{\'e}rentielle
  Cat{\'e}goriques}, 56:301--316, 2015.

\bibitem{cockett2021differential}
R.~Cockett, G.~Cruttwell, and J.-S.~P. Lemay.
\newblock Differential equations in a tangent category i: Complete vector
  fields, flows, and exponentials.
\newblock {\em Applied Categorical Structures}, pages 1--53, 2021.

\bibitem{CruttwellLucychynCohomology}
G.~Cruttwell and R.~Lucyshyn-Wright.
\newblock A simplicial foundation for differential and sector forms in tangent
  categories.
\newblock {\em Journal of Homotopy and Related Structures volume}, 13:867--925,
  2018.

\bibitem{cruttwellLemay:AlgebraicGeometry}
G.~S.~H. Cruttwell and Jean-Simon~Pacaud Lemay.
\newblock Differential bundles in commutative algebra and algebraic geometry.
\newblock {\em Theory Appl. Categ.}, 39:Paper No. 36, 1077--1120, 2023.

\bibitem{dokas04}
Ioannis Dokas.
\newblock Quillen-{B}arr-{B}eck (co-) homology for restricted {L}ie algebras.
\newblock {\em J. Pure Appl. Algebra}, 186(1):33--42, 2004.

\bibitem{dokastriplecohomology}
Ioannis Dokas.
\newblock Triple cohomology and divided powers algebras in prime
  characteristic.
\newblock {\em J. Aust. Math. Soc.}, 87(2):161--173, 2009.

\bibitem{dokas23}
Ioannis Dokas.
\newblock On {K}\"ahler differentials of divided power algebras.
\newblock {\em J. Homotopy Relat. Struct.}, 18(2-3):153--176, 2023.

\bibitem{dokasFI2025}
Ioannis Dokas, Martin Frankland, and Sacha Ikonicoff.
\newblock Quillen (co)homology of divided power algebras over an operad, 2025.

\bibitem{fox93}
Thomas~F. Fox.
\newblock An introduction to algebraic deformation theory.
\newblock {\em J. Pure Appl. Algebra}, 84(1):17--41, 1993.

\bibitem{GARNER2018668}
R.~Garner.
\newblock An embedding theorem for tangent categories.
\newblock {\em Advances in Mathematics}, 323:668--687, 2018.

\bibitem{grothendieck1966elements}
A.~Grothendieck.
\newblock {\'E}l{\'e}ments de g{\'e}om{\'e}trie alg{\'e}brique: Iv. {\'e}tude
  locale des sch{\'e}mas et des morphismes de sch{\'e}mas, troisi{\`e}me
  partie.
\newblock {\em Publications Math{\'e}matiques de l'IH{\'E}S}, 28:5--255, 1966.

\bibitem{ikonicoff2021cartesian}
S.~Ikonicoff and J.-S.~P. Lemay.
\newblock Cartesian differential comonads and new models of cartesian
  differential categories.
\newblock {\em arXiv preprint arXiv:2108.04304}, 2021.

\bibitem{ikonicofflemay23}
S.~Ikonicoff and J.-S.~Pacaud Lemay.
\newblock Cartesian differential comonads and new models of cartesian
  differential categories.
\newblock {\em Cahiers de topologie et géométrie différentielle
  catégoriques}, 64(2):198--239, 2023.

\bibitem{ikonicoffLL24}
Sacha Ikonicoff, Marcello Lanfranchi, and Jean-Simon~Pacaud Lemay.
\newblock The {R}osick\'y{} tangent categories of algebras over an operad.
\newblock {\em High. Struct.}, 8(2):332--385, 2024.

\bibitem{KP}
William~F. Keigher and F.~Leon Pritchard.
\newblock Hurwitz series as formal functions.
\newblock {\em J. Pure Appl. Algebra}, 146(3):291--304, 2000.

\bibitem{MacAdamVectorBundles}
B.~MacAdam.
\newblock Vector bundles and differential bundles in the category of smooth
  manifolds.
\newblock {\em Applied categorical structures}, 29(2):285--310, 2021.

\bibitem{roby63}
Norbert Roby.
\newblock Lois polynomes et lois formelles en th\'{e}orie des modules.
\newblock {\em Ann. Sci. \'{E}cole Norm. Sup. (3)}, 80:213--348, 1963.

\bibitem{roby65alg}
Norbert Roby.
\newblock Les alg\`ebres \`a puissances divis\'{e}es.
\newblock {\em Bull. Sci. Math. (2)}, 89:75--91, 1965.

\bibitem{RosickyTangentCats}
J.~Rosick{\`y}.
\newblock Abstract tangent functors.
\newblock {\em Diagrammes}, 12:JR1--JR11, 1984.

\bibitem{soublin87puissances}
Jean-Pierre Soublin.
\newblock Puissances divis\'{e}es en caract\'{e}ristique non nulle.
\newblock {\em J. Algebra}, 110(2):523--529, 1987.

\end{thebibliography}
\end{document}